\documentclass[11pt]{amsart}
\usepackage[top=30truemm,bottom=30truemm,left=30truemm,right=30truemm]{geometry} 
\usepackage{amsmath, amssymb, array}
\usepackage[all]{xy}
\usepackage{ascmac}
\usepackage[dvips]{graphicx}
\usepackage{color}
\usepackage{amscd}
\usepackage{amsrefs}
\usepackage{cleveref}
\usepackage{shuffle}
\usepackage{mathrsfs}
\usepackage{listings} 
\usepackage{comment}
\usepackage[OT2,T1]{fontenc}

\newfont{\cyr}{wncyr10 scaled\magstep0}

\newcommand{\F}{\mathbb{F}}
\newcommand{\G}{\mathbb{G}}

\newcommand{\Q}{\mathbb{Q}}
\newcommand{\R}{\mathbb{R}}

\newcommand{\Z}{\mathbb{Z}}

\newcommand{\mf}[1]{\mathfrak{#1}}

\DeclareMathOperator{\Aut}{Aut} 
\DeclareMathOperator{\Char}{char} 
\DeclareMathOperator{\Cl}{Cl} 
\DeclareMathOperator{\Gal}{Gal} 
\DeclareMathOperator{\Hom}{Hom} 
\DeclareMathOperator{\Ker}{Ker} 
\DeclareMathOperator{\length}{length} 
\DeclareMathOperator{\ord}{ord} 
\DeclareMathOperator{\rank}{rank} 
\DeclareMathOperator{\res}{res} 
\DeclareMathOperator{\Sel}{Sel} 
\DeclareMathOperator{\semisimple}{ss} 
\DeclareMathOperator{\tors}{tors} 
\DeclareMathOperator{\ur}{ur} 

\DeclareFontFamily{U}{wncy}{}
\DeclareFontShape{U}{wncy}{m}{n}{<->wncyr10}{}
\DeclareSymbolFont{mcy}{U}{wncy}{m}{n}
\DeclareMathSymbol{\Sha}{\mathord}{mcy}{"58}

\theoremstyle{plain}
 \newtheorem{theorem}{Theorem}[section]
 \crefname{theorem}{Theorem}{Theorems}
 \newtheorem{proposition}[theorem]{Proposition}
 \crefname{proposition}{Proposition}{Propositions}
 \newtheorem{lemma}[theorem]{Lemma}
 \crefname{lemma}{Lemma}{Lemmas}
 \newtheorem{corollary}[theorem]{Corollary}
 \crefname{corollary}{Corollary}{Corollaries}
 
 \crefname{conjecture}{Conjecture}{Conjectures}
 
 \crefname{hypothesis}{Hypothesis}{Hypotheses}
 \newtheorem{question}[theorem]{Question}
 \crefname{question}{Question}{Questions}
 
 \crefname{problem}{Problem}{Problems}
 
 \crefname{algorithm}{Algorithm}{Algorithms}

\theoremstyle{definition} 
 
 \crefname{definition}{Definition}{Definitions}
 
 \crefname{example}{Example}{Examples}
 \newtheorem{remark}[theorem]{Remark}
 \crefname{remark}{Remark}{Remarks}

\title{Elliptic analogue of irregular prime numbers for the $p^{n}$-division fields of the curves $y^{2} = x^{3}-(s^{4}+t^{2})x$}

\author{Naoto Dainobu}
\address[Naoto Dainobu]{Department of Mathematics \\ Faculty of Science and Technology \\ Keio University, 3-14-1, Hiyoshi, Kohoku, Yokohama, Japan}
\email{vicarious@keio.jp}
\author{Yoshinosuke Hirakawa}
\address[Yoshinosuke Hirakawa]{Department of Mathematics \\ Faculty of Science and Technology \\ Tokyo University of Science, 2641, Yamazaki, Noda, Chiba, Japan}
\email{hirakawa\_yoshinosuke@ma.noda.tus.ac.jp}
\author{Hideki Matsumura}
\address[Hideki Matsumura]{Department of Mathematics \\ Faculty of Science and Technology \\ Keio University, 3-14-1, Hiyoshi, Kohoku, Yokohama, Japan}
\email{hidekimatsumura@keio.jp}
\thanks{This research was supported by JSPS KAKENHI Grant Numbers 21J13502 and 21K13779.}
	\subjclass[2010]{primary 11G05; 
	secondary
		11R29,  	
		11R34,  	
		14G40.  	
		}
\date{\today}

\begin{document}


\maketitle

\begin{abstract}
A prime number $p$ is said to be irregular
if it divides the class number of the $p$-th cyclotomic field $\Q(\zeta_{p}) = \Q(\G_m[p])$.
In this paper, we study its elliptic analogue for the division fields of an elliptic curve.
More precisely,
for a prime number $p \geq 5$ and a positive integer $n$,
we study the $p$-divisibility of the class number of the $p^{n}$-division field $\Q(E[p^{n}])$ of an elliptic curve $E$ of the form $y^{2} = x^{3}-(s^{4}+t^{2})x$.
In particular, we construct a certain infinite subfamily consisting of curves
with novel properties that
they are of Mordell-Weil rank 1
and the class numbers of their $p^{n}$-division fields are divisible by $p^{2n}$.
Moreover,
we can prove that these division fields are not isomorphic to each other.
In our construction, we use recent results obtained by the first author.
\end{abstract}



\section{Introduction and main results}

\subsection{Background}

For every integer $N$ and an abelian group $A$,
we denote the subgroup of $A$ consisting of elements whose orders divide $N$ by $A[N]$.

Let $N$ be an integer,
$K$ be a field,
$\overline{K}$ be a fixed algebraic closure of $K$,
and $\mf{A}$ be a commutative group scheme of finite type (e.g.\ an elliptic curve $E$) defined over $K$.
For simplicity, suppose that $K$ is perfect.
Then, we have a natural Galois representation
\[
	\rho_{K, \mf{A}, n} : \Gal(\overline{K}/K) \to \Aut_{K}(\mf{A}(\overline{K})[N]).
\]
We define the $N$-division field $K(\mf{A}[N])$ of $\mf{A}$ over $K$ by the fixed field $K(\mf{A}[N]) := \overline{K}^{\Ker(\rho_{K, \mf{A}, n})}$.

One of the most classical examples of the $N$-division field is a cyclotomic field $\Q(\zeta_{N})$,
which corresponds to the case where $K = \Q$ and $\mf{A}$ is the multiplicative group scheme $\G_{m}$.
In this setting, a prime number $p$ is called \emph{irregular}
if the class number of $\Q(\zeta_{p}) = \Q(\G_{m}[p])$ is divisible by $p$
because Kummer's approach to Fermat's Last Theorem breaks down for irregular prime numbers \cite{Washington_text}.
\footnote{
	In fact, Kummer overcame, at least partially, such a difficulty arising from irregular prime numbers
	by some auxiliary method.
	However, since his methods depend on the \emph{degree} (or \emph{index}) of irregularity of each irregular prime number, such a bypass did not lead him to the complete proof of Fermat's Last Theorem.
	}
It has been an open question over 150 years
whether there exist infinitely many \emph{regular} (i.e., not irregular) prime numbers.
On the other hand,
it has been well-known that there exist infinitely many irregular prime numbers.
We refer the reader to Carlitz's elegant proof \cite{Carlitz_irregular},
which has a qualitative refinement \cite{Luca-PM-Pomerance}.
Moreover,
Ernvall \cite{Ernvall_1979,Ernvall_1983} generalized
the infinitude of prime numbers $p$ dividing the class numbers of $\Q(\G_{m}[p])$
to the infinitude of prime numbers $p$ (prime to a fixed $N$) dividing the class numbers of  $\Q(\G_{m}[Np])$
by means of generalized Bernoulli numbers.
These results lead us naturally to the following questions on \emph{elliptic analogue} of irregular prime numbers.

\begin{question} \label{deform_p}
Let $E$ be an elliptic curve defined over $\Q$.
Then, how often the class numbers of the $p$-division fields $\Q(E[p])$ are divisible by $p$
when $p$ runs over the prime numbers?
\end{question}

Recent results by Ray and Weston \cite[Theorems 5.1 and 5.2]{Ray-Weston}
give partial answers to this question.
On the other hand,
since there are infinitely many non-isomorphic elliptic curves over $\Q$,
it is also natural to ask the following variant of \cref{deform_p}.

\begin{question} \label{deform_E}
Let $p$ be a prime number.
Then, how often the class numbers of the $p$-division fields $\Q(E[p])$ are divisible by $p$
when $E$ runs over the elliptic curves over $\Q$?
\end{question}

To \cref{deform_E},
a recent conditional result by Ray and Weston \cite[Theorem 4.4]{Ray-Weston} in the same paper
gives an effective lower bound for the density $\underline{\mathfrak{d}}(\mathcal{E}_{p})$ of elliptic curves $E$
for which $\Hom_{\Gal(\Q(E[p])/\Q)}(\Cl(\Q(E[p])), E[p]) = 0$ for each odd prime number $p$.
Here and after, $\Cl(K)$ denotes the ideal class group of a number field $K$.
There are also unconditional results for small prime numbers $p$.
For $p = 2$,
we have some qualitative results to \cref{deform_E}
e.g.\ each of \cite{Bhargava_quartic_order,Davenport-Heilbronn_II,Ho-Shankar-Varma}.
A similar result can be deduced by combining \cite[Theorem 1.4]{Li_Trans} and \cite[Theorem 1.1]{Bhargava-Shankar_2Sel}.
In fact, the last result \cite[Theorem 1.1]{Bhargava-Shankar_2Sel} is generalized for $p = 3$ and $5$ 
in \cite{Bhargava-Shankar_3Sel} and \cite{Bhargava-Shankar_5Sel} respectively,
hence it might be possible to obtain similar $p$-divisibility results of the class numbers also for these small $p$ along this line unconditionally.

Both of the \cref{deform_p,deform_E} ask the stochastic aspects
of the $p$-divisibility of the class number of $\Q(E[p])$.
On the other hand,
there are several recent results on the deterministic aspects of
the growth of the $p$-part of $\Cl(\Q(E[p^{n}]))$ along $n \to \infty$.
For example,
the preceding works
\cite[Theorem 1.1]{Hiranouchi},
\cite[Theorem 1.1]{Sairaiji-Yamauchi_JNT}, and
\cite[Theorem 1.1]{Sairaiji-Yamauchi_JTNB}
give sufficient conditions on triplet $(E, p, n)$
so that the class number of $\Q(E[p^{n}])$ is divisible by a power of $p$.
More precisely,
all of their main results have the following form
\[
	\ord_{p}(\#\Cl(\Q(E[p^{n}]))) \geq 2n\rank E(\Q) - \text{(local contributions)}.
\]
In fact, these general bounds can imply the $p$-divisibility of the class number of $\Q(E[p])$ only if $\rank E(\Q) \geq 2$
because of the above local contributions.
Hence,
if we are interested in elliptic curves $E$ such that $\rank E(\Q) \leq 1$, 
we need a different approach.
Moreover, it is fair to say that the condition $\rank E(\Q) \geq 2$ is too restrictive at least hypothetically
because the conjecture of Goldfeld \cite{Goldfeld} and Katz-Sarnak \cite[\S5]{Katz-Sarnak}
predicts that a half of elliptic curves over $\Q$ have rank 0 and the other have rank 1.
Thus,
it is natural to focus on an infinite family of elliptic curves $E$ such that $\rank E(\Q) \leq 1$
especially for $p \geq 7$.

In this paper,
for any arbitrarily given prime power $p^{n}$ with $p \geq 5$,
we construct an explicit infinite family consisting of elliptic curves $E$ such that
$\rank E(\Q) = 1$ and the class numbers of $\Q(E[p^{n}])$ are divisible by $p^{2n}$.

\subsection{Main results}

Let $s, t \in \Z$ be integers such that $(s, t) \neq (0, 0)$
and $E_{s, t}$ be an elliptic curve defined by $y^{2} = x^{3}-(s^{4}+t^{2})x$.
The first main result of this paper is the following:

\begin{theorem} \label{infinite}
Let $p$ be a prime number and $n$ be an integer.
Suppose that $p \geq 5$.
Then, there exist infinitely many prime numbers of the form $l = s^{4}+t^{2}$ with $s, t \in \Z$
such that $\rank E_{s, t}(\Q) = 1$ and the class number of $\Q(E_{s, t}[p^{n}])$ is divisible by $p^{2n}$.
Moreover, these number fields $\Q(E_{s, t}[p^{n}])$ and $\Q(E_{s', t'}[p^{n}])$ are isomorphic to each other
if and only if the prime number $s^{4}+t^{2}$ coincides with the other ${s'}^{4}+{t'}^{2}$.
\end{theorem}

The proof is based on the following theorem,
which is the second main result of this paper.

\begin{theorem} \label{main_result}
Let $p$ be a prime number and $n$ be an integer.
Suppose that $p \geq 5$,
$\gcd(s, t) = 1$,
$st \equiv 0 \bmod{p^{n+1}}$,
and $s^{4}+t^{2}$ is fourth-power-free and not a square.
Then, there exists a non-zero $\Gal(\Q(E_{s, t}[p])/\Q)$-equivariant homomorphism
\[
	\Cl(\Q(E_{s, t}[p^{n}])) \to E_{s, t}[p^{n}].
\]
Moreover, the class number of $\Q(E_{s, t}[p^{n}])$ is divisible by $p^{2n}$.
\end{theorem}

The above results show the divisibility of the order of the ideal class group $\Cl(\Q(E_{s, t}[p^{n}]))$.
Since it admits a natural Galois action,
it is also natural to study the Galois module structure of $\Cl(\Q(E_{s, t}[p^{n}]))$.
For $n = 1$, we give the following theorem,
which is our third main result.

\begin{theorem} \label{refine_Fujita-Terai}
Let $p$ be a prime number.
Suppose that $p \geq 5$,
$t = \tau^{2}$ ($\tau \in \Z$),
$s^{4}+\tau^{4}$ is fourth-power-free,
$s\tau \equiv 0 \bmod{p^{2}}$,
and $E[p]$ is an irreducible $\F_{p}[\Gal(\Q(E_{s, \tau^{2}}[p])/\Q)]$-module.
Then, there exists a surjective $\Gal(\Q(E_{s, \tau^{2}}[p])/\Q)$-equivariant homomorphism
\[
	\left( \frac{\Cl(\Q(E_{s, \tau^{2}}[p]))}{p\Cl(\Q(E_{s, \tau^{2}}[p]))} \right)^{\semisimple}
	\to E_{s, \tau^{2}}[p]^{\oplus 2}.
\]
In particular, the class number of $\Q(E_{s, \tau^{2}}[p])$ is divisible by $p^{4}$.
\end{theorem}

Let us explain the outline of this paper.
In \S2, we review the previous work of the first author \cite{Dainobu_preparation},
which gives a lower bound for the number of Galois equivariant morphisms of $\Cl(\Q(E[p^{n}]))$ to $E[p^{n}]$
under certain technical conditions on an elliptic curve $E$.
After that, in \S3, we verify these conditions for our elliptic curves $E_{s, t}$,
which completes the proof of \cref{main_result}.
A key point is that
a recent work of \cite{Voutier-Yabuta_quartic} ensures that
an obvious rational point $P_{s, t} := (-s^{2}, st)$ on $E_{s, t}$
does not fall into a subgroup $pE_{s, t}(\Q)$ of $E_{s, t}(\Q)$,
which ensures the most important necessary condition in \cite{Dainobu_preparation}.
In \S4,
we deduce \cref{infinite} from \cref{main_result}.
This will be done by a collaboration of two monumental works in number theory;
one is the N\'eron-Ogg-Shafarevich criterion \cite{Serre-Tate} for good reduction of elliptic curves
and the other is the Friedlander-Iwaniec's theorem \cite{Friedlander-Iwaniec} for prime numbers of the form $s^{4}+t^{2}$ with integers $s, t$.
Finally, in \S5,
we give a proof of the fact that $\rank E_{s, t}(\Q) = 1$ if $s^{4}+t^{2}$ is a prime number congruent to $9 \bmod{16}$.
This implies that for any arbitrarily given prime power $p^{n}$ with $p \geq 5$,
we can construct an explicit infinite family consisting of elliptic curves $E$ such that
$\rank E(\Q) = 1$,
their $p^{n}$-division fields $\Q(E[p^{n}])$ are non-isomorphic to each other,
and the class numbers of $\Q(E[p^{n}])$ are divisible by $p^{2n}$.
Although there are several preceding works
(e.g.\ \cite{Sairaiji-Yamauchi_JNT,Sairaiji-Yamauchi_JTNB,Hiranouchi} mentioned above and \cite{Ohshita,Prasad-Shekher})
on the $p$-divisibility of the class numbers of $\Q(E[p^{n}])$,
as far as the authors know,
there is no literature including an infinite family of elliptic curves having the above properties.

\section{A sufficient condition for $p$-divisibility of the class number}

Let $i$ be a non-negative integer.
For every separable extension $L/K$ of fields
and a $\Gal(L/K)$-module $M$,
we abbreviate the Galois cohomology group $H^{i}(\Gal(L/K), M)$ by $H^{i}(L/K, M)$.
Moreover, if $L$ is a separable closure of $K$,
then we abbreviate $H^{i}(L/K, M)$ by $H^{i}(K, M)$.
For general properties of Galois cohomology groups,
we refer the reader to \cite{NSW}.

In what follows, we assume that $K$ is a finite extension of $\Q$.

Let $\mathfrak{A}$ be a commutative group scheme of finite type defined over $K$,
and $N$ be an integer.
We are interested in the ideal class group $\Cl(K(\mathfrak{A}[N]))$.
By the class field theory (see e.g.\ \cite{Janusz,NSW}),
we have a canonical isomorphism $\Cl(K(\mathfrak{A}[N])) \simeq \Gal(K(\mathfrak{A}[N])^{\ur}/K(\mathfrak{A}[N]))$,
where $K(\mathfrak{A}[N])^{\ur}$ denotes the maximal unramified abelian extension of $K(\mathfrak{A}[N])$.
Therefore,
in order to obtain information of $\Cl(K(\mathfrak{A}[N]))$,
it is natural to study the subgroup of $\Hom(\Gal(\overline{K}/K(\mathfrak{A}[N])), A)$
consisting of unramified homomorphisms for several (or all if possible) finite abelian groups $A$.
Moreover,
since the group $\Cl(K(\mathfrak{A}[N]))$ admits a natural action of the Galois group $\Gal(K(\mathfrak{A}[N])/K)$,
it is more natural to study the subgroup
\[
	\Hom_{\Gal(K(\mathfrak{A}[N])/K)}(\Gal(\overline{K}/K(\mathfrak{A}[N])), M)
	\simeq H^{1}(K(\mathfrak{A}[N]), M)^{\Gal(K(\mathfrak{A}[N])/K)}
\]
consisting of unramified elements for several (or all if possible) finite $\Gal(K(\mathfrak{A}[N])/K)$-modules $M$.
In particular,
it is interesting, at least as a first step,
to carry out the above plan for the $\Gal(K(\mathfrak{A}[N])/K)$-module $M = \mathfrak{A}[N]$.

By the definition of the $\Gal(K(\mathfrak{A}[N])/K)$-module $\mathfrak{A}[N]$,
we have an exact sequence of $\Gal(\overline{K}/K)$-modules
\[
	0 \to \mathfrak{A}[N] \to \mathfrak{A}(\overline{K}) \xrightarrow{N} \mathfrak{A}(\overline{K}).
\]
In what follows,
we assume that the last map $N : \mathfrak{A}(\overline{K}) \to \mathfrak{A}(\overline{K})$ is surjective.
Then, the above exact sequence yields a short exact sequence of the Galois cohomology groups,
the so called Kummer sequence:
\[
	0 \to \frac{\mathfrak{A}(K)}{N\mathfrak{A}(K)}
	\xrightarrow{\kappa = \kappa_{\mathfrak{A}, N}} H^{1}(K, \mathfrak{A}[N])
		\to H^{1}(K, \mathfrak{A})[N] \to 0.
\]
On the other hand,
the following inflation-restriction exact sequence is one of the most elementary tools in the theory of Galois cohomology groups:
\[
	1
	\to H^{1}(K(\mathfrak{A}[N])/K, \mathfrak{A}[N])
		\xrightarrow{\inf} H^{1}(K, \mathfrak{A}[N])
	\xrightarrow{\res} H^{1}(K(\mathfrak{A}[N]), \mathfrak{A}[N])^{\Gal(K(\mathfrak{A}[N])/K)}
		\to \cdots
\]
Therefore,
we obtain a natural composite homomorphism induced by the Kummer map $\kappa$ and the restriction map $\res$
\[
	\frac{\mathfrak{A}(K)}{N\mathfrak{A}(K)}
	\to \frac{H^{1}(K, \mathfrak{A}[N])}{H^{1}(K(\mathfrak{A}[N])/K, \mathfrak{A}[N])} \hookrightarrow \Hom_{\Gal(K(\mathfrak{A}[N])/K)}(\Gal(\overline{K}/K(\mathfrak{A}[N])), \mathfrak{A}[N]).
\]
This suggests that
we can construct non-zero elements in the right-most side
by using non-zero elements in the left-most side.
In this view point,
it is clear that $H^{1}(K(\mathfrak{A}[N])/K, \mathfrak{A}[N])$ plays a role as the obstruction for the transformation of the non-zero elements in $\mathfrak{A}(K)/N\mathfrak{A}(K)$
to non-zero homomorphisms.

On the other hand,
since we are interested in unramified homomorphisms of $\Gal(\overline{K}/K(\mathfrak{A}[N]))$ to $\mathfrak{A}[N]$,
we want to check that which elements in $\mathfrak{A}(K)/N\mathfrak{A}(K)$ correspond to unramified homomorphisms.
This motivates us to define
\[
	\mathfrak{A}(K)_{\ur, N}
	= \Ker\left( \mathfrak{A}(K) \to \prod_{v} \frac{\mathfrak{A}(K_{v}^{\ur})}{N\mathfrak{A}(K_{v}^{\ur})} \right).
\]
Along this line,
in a forthcoming paper \cite{Dainobu_preparation},
the first author studied some conditions under which
$\res(\kappa(\mathfrak{A}(K)_{\ur, N})) \neq 0$
for a fixed elliptic curve $\mathfrak{A} = E$ defined over $K = \Q$ and a fixed prime power $N = p^{n}$.
The results are summarized as follows:

\begin{theorem} [{\cite{Dainobu_preparation}}] \label{Dainobu_criterion}
Let $E$ be an elliptic curve defined over $\Q$,
$p$ be a prime number,
and $n$ be a positive integer.
Suppose that the following conditions hold.
\begin{enumerate}
\item
$E$ has good reduction at $p$.

\item
$p \geq 5$ and the denominator of $j(E)$ is $p$-th power free.

\item
$H^{1}(\Q(E[p^{n}])/\Q, E[p^{n}]) = 0$.
\end{enumerate}
Then, we have $E(\Q)_{\ur, p^{n}} = E(\Q) \cap p^{n}E(\Q_{p})$.
Moreover, if we denote the length of $E(\Q)_{\ur, p^{n}}/p^{n}E(\Q)$ as a $\Z_{p}$-module by $r_{\ur, p^{n}}$.
Then, the following inequality holds.
\[
	\length_{\Z_{p}} \Hom_{\Gal(K(E[p^{n}])/\Q)}(\Cl(\Q(E[p^{n}])), E[p^{n}]) \geq r_{\ur, p^{n}}.
\]
\end{theorem}

\begin{theorem} [multiplicity bound {\cite{Dainobu_preparation}}] \label{Dainobu_multiplicity}
In the setting of \cref{Dainobu_criterion},
suppose that $n = 1$ and $E[p]$ is an irreducible $\F_{p}[\Gal(\Q(E[p])/\Q)]$-module.
Then, there exists a $\Gal(\Q(E[p])/\Q)$-equivariant surjective homomorphism
\[
	\left( \frac{\Cl(\Q(E[p]))}{p\Cl(\Q(E[p]))} \right)^{\semisimple} \to E[p]^{\oplus r_{\ur, p}},
\]
whose domain is the semi-simplification of $\Cl(\Q(E[p]))/p\Cl(\Q(E[p]))$
as an $\F_{p}[\Gal(\Q(E[p])/\Q)]$-module.
\end{theorem}

In application,
the following corollary of \cref{Dainobu_criterion} is valuable.

\begin{corollary} [{\cite{Dainobu_preparation}, see also \cite[Lemma 2.10]{Ohshita}}] \label{Dainobu_corollary}
Assume the same conditions as \cref{Dainobu_criterion}.
\begin{enumerate}
\item
Suppose that $E$ has a Weirstrass equation $y^{2} = x^{3}+ax+b$ over $\Z_{p}$
and there exists a rational point $P \in E(\Q) \setminus p^{n}E(\Q)$ and a divisor $f_{0}$ of $\#E(\F_{p})$ prime to $p$
such that $x(f_{0}P) \not\in \Z_{p}$ and $x(f_{0}P)/y(f_{0}P) \in p^{n}\Z_{p}$.
Then, the class number of $\Q(E[p^{n}])$ is divisible by $p$.

\item
Suppose that $E[p]$ is an irreducible $\F_{p}[\Gal(\Q(E[p])/\Q)]$-module.
Then, the class number of $\Q(E[p^{n}])$ is divisible by $p^{2r_{\ur, p^{n}}}$.
\end{enumerate}
\end{corollary}

\begin{remark}
\begin{enumerate}
\item
In \cite{Lawson-Wuthrich},
the triples $(E, p, n)$ for which the third condition
\[
	H^{1}(\Q(E[p^{n}])/\Q, E[p^{n}]) = 0
\]
of \cref{Dainobu_criterion} fails
are classified in terms of isogeny on $E$ of degree $p$.
In particular, the third condition holds for all $(E, p, n)$ such that $p \geq 13$.
Moreover it holds for all $(E, 11, n)$ (resp.\ $(E, 7, n), (E, 5, n)$)
whenever $E$ is not 121c1 nor 121c2 in Cremona's table \cite{Cremona_text}
(resp.\ $E(\Q)[7] = 0$, $E[5]$ is an irreducible $\F_{5}[\Gal(\overline{\Q}/\Q)]$-module and the quadratic twist of $E$ by $D = 5$ has no rational points of order 5).

\item
For the convenience of the reader,
we give a sketch of the proof of \cref{Dainobu_criterion}.
As we have already seen,
the third condition is equivalent to say that the restriction map
\begin{align*}
	H^{1}(\Q, E[p^{n}])
	&\xrightarrow{\res} H^{1}(\Q(E[p^{n}]), E[p^{n}])^{\Gal(\Q(E[p^{n}])/\Q)} \\
	&\simeq \Hom_{\Gal(\Q(E[p^{n}])/\Q)}(\Gal(\overline{\Q}/\Q(E[p^{n}])), E[p^{n}])
\end{align*}
is injective.
The first condition ensures that
the image of the set $E(\Q)_{\ur, n}$ by the mod $p^{n}$ Kummer map $\kappa : E(\Q) \to H^{1}(\Q, E[p^{n}])$ consists of 1-cocycle classes unramified at $p$.
Finally, the second condition ensures that
every 1-cocycle classes in $H^{1}(\Q, E[p^{n}])$ is unramified at every place $v \neq p$.
\footnote{
	More precisely,
	every 1-cocycle is unramified at the infinite place
	because $p$ is odd.
	}

\item
\cref{Dainobu_multiplicity} means that
if one decomposes the semi-simplification of
the $\F_{p}[\Gal(\Q(E[p])/\Q)]$-module $\Cl(\Q(E[p]))/p\Cl(\Q(E[p]))$ into its irreducible components,
then the module $E[p]$ appears at least $r_{\ur, 1}$-times.
Therefore, we can obtain an explicit lower bound for the multiplicity of this special component contained in our object $\Cl(\Q(E[p]))/p\Cl(\Q(E[p]))$
as the name of \cref{Dainobu_multiplicity} suggests.

\item
\cref{Dainobu_corollary}(1) gives a useful condition
to prove the $p$-divisibility of the class number of $\Q(E[p^{n}])$
if one knows a specific rational point on $E$.
On the other hand,
\cref{Dainobu_corollary}(2) gives a useful condition
to prove that the class number of $\Q(E[p^{n}])$ is divisible by a higher power of $p$
if one knows a better lower bound for $r_{\ur, p^{n}}$.
\end{enumerate}
\end{remark}

Now, let us explain what the above theorems give for the proof of \cref{main_result}.
Since we assume that $st \equiv 0 \bmod{p}$ and $s^{4}+t^{2}$ is square-free,
we see that $s^{4}+t^{2}$ is prime to $p$.
Hence, the first condition holds for $(E, p, n) = (E_{s, t}, p, 1)$ because
$E_{s, t}$ has good reduction outside $2(s^{4}+t^{2})$.
Moreover, the second condition holds for our elliptic curves $E_{s, t}$ because $j(E_{s, t}) = 2^{6} \cdot 3^{3}$.
Finally, we can verify that $H^{1}(\Q, E_{s, t}[p^{n}]) = 0$
by checking the following facts.
\begin{enumerate}
\item
For $p = 11$, $E_{s, t} \not\simeq \mathrm{121c1, 121c2}$ in Cremona's table \cite{Cremona_text}.

\item
For $p = 7$, $E_{s, t}(\Q)[7] = 0$ by \cite[Ch.\ X, Proposition 6]{Silverman_AEC}.

\item
For $p = 5$, $E_{s, t}[5]$ is an irreducible $\F_{5}[\Gal(\Q(E_{s, t}[5])/\Q)]$-module by \cite[Theorem 7]{DGJJU}, and the quadratic twist of $E_{s, t}$ by $D = 5$ is given by $y^2 = x^{3}-25(s^{4}+t^{2})x$ and has no rational 5-torsion point by \cite[Ch.\ X, Proposition 6]{Silverman_AEC}.
\end{enumerate}
Thus, thanks to \cref{Dainobu_corollary}(1),
it is sufficient for the proof of \cref{main_result} to check the followings.
\begin{enumerate}
\item
If $s^{4}+t^{2}$ is fourth-power-free and not a square,
then an obvious rational point $P_{s, t} := (-s^{2}, st) \in E_{s, t}(\Q)$ does not lie in the subgroup $pE_{s, t}(\Q)$.

\item
If $p$ divides exactly one of $s$ and $t$ and $p^{n+1}$ divides $st$,
then $x(2P_{s, t}) \not\in p^{n+1}\Z_{p}$ and $x(2P_{s, t})/y(2P_{s, t}) \in p^{n+1}\Z_{p}$.
Here, note that since the group $E_{s, t}(\F_{p})$ of the modulo $p$ rational points contains a 2-torsion point $(0, 0)$, the order $\#E_{s, t}(\F_{p})$ is even.
\end{enumerate}

In the next section,
we verify these conditions.

\section{Proof of \cref{main_result,refine_Fujita-Terai}}

\subsection{Primitivity of the obvious rational point}

The goal of this subsection is to prove the following:

\begin{theorem} [{cf.\ \cite[Theorem 10.1]{Duquesne_quartic}, \cite[Theorem 1.5(1)]{Fujita-Terai}}] \label{primitivity}
Let $E_{s, t}$ be an elliptic curve defined by $y^{2} = x^{3}-(s^{4}+t^{2})x$ with $s, t \in \Z_{\geq 1}$
and $P_{s, t} := (-s^{2}, st)$ be a rational point on $E_{s, t}$.
Suppose that $s^{4}+t^{2}$ is forth-power free and not a square.
Then, the rational point $P_{s, t}$ can be extended to a minimal system of generators of $E_{s, t}(\Q)$.
\end{theorem}

\begin{remark}
If $s^{4}+t^{2}$ is a prime number congruent to $9 \bmod{16}$,
then $E_{s, t}(\Q)$ is generated by $(0, 0)$ and $P_{s, t}$.
For the proof, see \S5.
\end{remark}

The proof of \cref{primitivity} is based on some inequalities on height of rational points on $E_{s, t}$.
Before the proof,
we recall these inequalities with some terminologies on height.

First, we recall two kinds of \emph{heights} of a rational point $P$ on an elliptic curve $E$.

Suppose that $E$ is defined by $y^{2} = x^{3}+ax+b$ with $a, b \in \Z$.
Let $n, d \in \Z$ such that $x(P) = n/d$ and $\gcd(n, d) = 1$.
Then, the \emph{absolute logarithmic height} $h(n/d)$
of the rational number $n/d$ is defined by
\[
	h(n/d)
	:= \sum_{\text{$v$ : place of $\Q$}} \log^{+}\left| n/d \right|_{v}
	= \log\max\left( \left| n \right|, \left| d \right| \right),
\]
where $\log^{+}(\alpha) := \log\max\{ 1, \alpha \}$,
and the \emph{naive} (or \emph{Weil}) \emph{height} $h(P)$ of the rational point $P$ is defined by
\[
	h(P) := h(x(P)) = h(n/d).
\]
The naive height of a given point is easy to calculate,
but it is non-canonical
in the sense that it depends on a fixed defining equation of $E$.

On the other hand,
it is known \cite[Proposition 9.1]{Silverman_AEC} that the limit
\[
	\hat{h}(P)
	= \frac{1}{\deg(x)}\lim_{k \to \infty} \frac{h(kP)}{k^{2}}
	= \frac{1}{2}\lim_{n \to \infty} \frac{h(2^{n}P)}{4^{n}}
\]
exists and defines a positive definite quadratic form on the $\R$-vector space $\R \otimes E(\Q)$.
We call the function $\hat{h} : E(\Q) \to \R$ as the \emph{canonical} (or \emph{N\'eron-Tate}) \emph{height} on $E$.

\begin{remark}
Here, we follow the definition of the canonical height in \cite[Ch.\ VIII.6]{Silverman_AEC} and \cite{Silverman_difference}.
It should be remarked that
several authors use the other normalization of the canonical height.
For example,
in \cite{Cremona-Prickett-Siksek,Duquesne_quartic,Fujita-Terai},
the canonical height is defined as the twice of $\hat{h}$.
For such a discordance of the ``canonical" heights (and of the local heights),
the contents of \cite[\S4]{Cremona-Prickett-Siksek} is worth reading.
\end{remark}

The following inequalities give bounds for the difference of these two kinds of heights.

\begin{theorem} [{\cite[Theorem 1.1]{Silverman_difference}, \cite[Proposition 8.1]{Duquesne_quartic}}] \label{Silverman_inequality}
Let $E$ be an elliptic curve defined by $y^{2} = x^{3}+ax+b$ with $a, b \in \Z$,
$\Delta(E)$ be its discriminant, $j(E)$ be its $j$-invariant,
and $P \neq \infty$ be a rational point on $E$.
Then, the following inequalities hold:
\[
	-\frac{h(j(E))}{8}-\frac{h(\Delta(E))}{12}-0.973
	\leq \hat{h}(P)-\frac{1}{2}h(P)
	\leq \frac{h(j(E))}{12}+\frac{h(\Delta(E))}{12}+1.07.
\]
In particular,
if $a = -(s^{4}+t^{2})$ with $s, t \in \Z$ and $b = 0$, 
then we have the upper bound for the canonical height of $P$ as follows:
\[
	\hat{h}(P) \leq \frac{1}{2}h(P)+\frac{1}{4}\log(s^{4}+t^{2})+2.03781.
\]
\end{theorem}

Therefore, if we obtain an explicit lower bound for the canonical height,
then we can obtain an explicit lower bound for the naive height.

\begin{theorem} [{\cite[Theorem 1.2]{Voutier-Yabuta_quartic}}] \label{Voutier-Yabuta}
Let $E_{a}$ be an elliptic curve defined by $y^{2} = x^{3}+ax$ with a forth-power free integer $a$
and $Q \in E_{a}(\Q)$ be a non-torsion point.
Then, we have the following estimate:
\[
	\hat{h}(Q)
	> \frac{1}{16}\log\left| a \right|
	+ \begin{cases}
	\frac{1}{2}\log2 & \text{if $a > 0$ and $a \equiv 1, 5, 7, 9, 13, 15 \bmod{16}$} \\
	\frac{1}{4}\log2 & \text{if $a > 0$ and either $a \equiv 20, 36 \bmod{64}$,} \\
				& \text{or $a \equiv 2, 3, 6, 8, 10, 11, 12, 14 \bmod{16}$} \\
	-\frac{1}{8}\log2 & \text{if $a > 0$ and $a \equiv 4, 52 \bmod{64}$} \\
	\frac{9}{16}\log2 & \text{if $a < 0$ and $a \equiv 1, 5, 7, 9, 13, 15 \bmod{16}$} \\
	\frac{5}{16}\log2 & \text{if $a < 0$ and either $a \equiv 20, 36 \bmod{64}$} \\
				& \text{or $a \equiv 2, 3, 6, 8, 10, 11, 12, 14 \bmod{16}$} \\
	-\frac{1}{16}\log2 & \text{if $a < 0$ and $a \equiv 4, 52 \bmod{64}$}
	\end{cases}
\]
\end{theorem}

Now, we can prove \cref{primitivity}.

\begin{proof} [Proof of \cref{primitivity}]
Since we assume that $s^{4}+t^{2}$ is not a square,
\cite[Ch.\ X, Proposition 6.1(a)]{Silverman_AEC} shows that
\[
	E_{s, t}(\Q)_{\tors} = \langle (0, 0) \rangle \simeq \Z/2\Z.
\]
Thus, it is sufficient to prove that
if there exists a rational point $Q \in E_{s, t}(\Q)$, a positive integer $n$, and $\delta \in \{ 0, 1 \}$ such that
\[
	P_{s, t} = nQ+\delta(0, 0),
\]
then $n = 1$.
Here, note that $Q \neq (0, 0), \infty$ because $s \neq 0$.

First, we prove that $n \neq 2$.
Indeed, if $n = 2$ and $\delta = 0$,
then the duplication formula implies that
\[
	-s^{2}
	= x(2Q)
	= \left( \frac{x(Q)^{2}+s^{4}+t^{2}}{2y(Q)} \right)^{2},
\]
which is impossible for $s \neq 0$.
On the other hand,
if $n = 2$ and $\delta = 1$,
then the addition formula implies that
\[
	-s^{2}
	= x(2Q+(0, 0))
	= -\frac{s^{4}+t^{2}}{x(2Q)},
	\quad \text{i.e.,} \quad
	\frac{s^{4}+t^{2}}{s^{2}}
	= x(2Q)
	= \left( \frac{x(Q)^{2}+s^{4}+t^{2}}{2y(Q)} \right)^{2},
\]
which is impossible for non-square $s^{4}+t^{2}$.

Next, since 
\[
	\frac{\hat{h}(P_{s, t})}{\hat{h}(Q)} = n^{2},
\]
it is sufficient to prove that the left hand side is smaller than 9.
Suppose that $s^{4}+t^{2}$ is odd.
Then, since $s^{4}+t^{2} \equiv 1 \bmod{8}$,
\cref{Silverman_inequality,Voutier-Yabuta} imply the desired bound
\[
	\frac{\hat{h}(P_{s, t})}{\hat{h}(Q)}
	\leq \frac{\frac{1}{4}\log(s^{4}+t^{2})+\frac{1}{2}\log(s^{2})+2.03781}{\frac{1}{16}\log(s^{4}+t^{2})+\frac{9}{16}\log2}
	< 9,
\]
where the second inequality follows from
\[
	\frac{1}{2}\log(s^{4}+t^{2})+2.03781
	< \frac{9}{16}\log(s^{4}+t^{2})+3.50905,
	\quad \text{i.e.,} \quad
	s^{4}+t^{2} > e^{-1.47124 \times 16}. 
\]
Suppose that $s^{4}+t^{2}$ is even.
If $s^{4}+t^{2} = 2$,
then we can check the assertion e.g.\ by using MAGMA command \texttt{Generators}.
If $s^{4}+t^{2} \neq 2$ and even,
then since we assume that $s^{4}+t^{2}$ is not a square,
we have $s^{4}+t^{2} \geq 6$.
Moreover, since we assume that $s^{4}+t^{2}$ is not divisible by $2^{4}$,
we see that $s^{4}+t^{2} \equiv 2 \bmod{4}$ or $s^{4}+t^{2} \equiv 4, 20 \bmod{32}$.
Therefore, \cref{Voutier-Yabuta} implies the desired bound
\[
	\frac{\hat{h}(P_{s, t})}{\hat{h}(Q)}
	\leq \frac{\frac{1}{4}\log(s^{4}+t^{2})+\frac{1}{2}\log(s^{2})+2.03781}{\frac{1}{16}\log(s^{4}+t^{2})+\frac{5}{16}\log2}
	< 9
\]
where the second inequality follows from
\[
	\frac{1}{2}\log(s^{4}+t^{2})+2.03781
	< \frac{9}{16}\log(s^{4}+t^{2})+1.94947,
	\quad \text{i.e.,} \quad
	s^{4}+t^{2} > e^{0.08834 \times 16} \approx 4.1.
\]
This completes the proof of \cref{primitivity}.
\end{proof}

\subsection{Unramifiedness at $p$}

The goal of this subsection is to prove the following:

\begin{proposition} \label{unramified_at_p}
Let $E_{s, t}$ be an elliptic curve defined by $y^{2} = x^{3}-(s^{4}+t^{2})x$ with $s, t \in \Z$,
$p$ be an odd prime number,
and $n$ be a positive integer.
Suppose that
\begin{enumerate}
\item
$p$ divides exactly one of $s$ and $t$.

\item
$p^{n+1}$ divides $st$.
\end{enumerate}
Then, $2P_{s, t} \in p^{n}E_{s, t}(\Q_{p})$.
\end{proposition}

\begin{proof}
By the duplication formula,
we have
\[
	x(2P_{s, t})
	= \left( \frac{2s^{4}+t^{2}}{2st} \right)^{2}.
\]
Moreover, the assumptions implies that
\[
	v_{p}(x(2P_{s, t}))
	= -2v_{p}(st) \leq -2(n+1) < 0,
\]
hence
\[
	v_{p}(y(2P_{s, t}))
	= \frac{1}{2}v_{p}(x(2P_{s, t})^{3} - (s^{4}+t^{2})x(P_{2s, t}))
	= \frac{3}{2}v_{p}(x(2P_{s, t}))
	= -3v_{p}(st) < 0.
\]
In particular, we have $2P_{s, t} \in \Ker(E_{s, t}(\Q_{p}) \to E_{s, t}(\F_{p}))$ and
\[
	v_{p}\left( \frac{x(2P_{s, t})}{y(2P_{s, t})} \right)
	= v_{p}(st) \geq n+1.
\]
On the other hand,
\cite[Ch.\ IV, Theorem 6.4]{Silverman_AEC} gives an isomorphism
\[	
	z : \Ker(E_{s, t}(\Q_{p}) \to E_{s, t}(\F_{p})) \simeq p\Z_{p}; (x, y) \mapsto -\frac{x}{y}.
\]
In particular,
we see that $2P_{s, t} \in p^{n}E_{s, t}(\Q_{p})$.
Since $p$ is odd,
we obtain the assertion $P_{s, t} \in p^{n}E_{s, t}(\Q_{p})$.
This completes the proof.
\end{proof}

\subsection{Proof of \cref{main_result}}

\begin{proof}
As we have seen in the last of \S2,
it is sufficient to check the following two conditions.
\begin{enumerate}
\item
If $s^{4}+t^{2}$ is fourth-power-free and not a square,
then an obvious rational point $P_{s, t} := (-s^{2}, st) \in E_{s, t}(\Q)$ does not lie in the subgroup $pE_{s, t}(\Q)$.

\item
If $p$ divides exactly one of $s$ and $t$ and $p^{n+1}$ divides $st$,
then $x(2P_{s, t}) \not\in p^{n+1}\Z_{p}$ and $x(2P_{s, t})/y(2P_{s, t}) \in p^{n+1}\Z_{p}$.
Here, note that since the group $E_{s, t}(\F_{p})$ of the modulo $p$ rational points contains a 2-torsion point $(0, 0)$, the order $\#E_{s, t}(\F_{p})$ is even.
\end{enumerate}
Each of them follows immediately from \cref{primitivity,unramified_at_p} respectively.
\end{proof}

\subsection{Proof of \cref{refine_Fujita-Terai}}

\begin{proof} [Proof of \cref{refine_Fujita-Terai}]
Suppose that $t$ is a square, say $t = \tau^{2}$ for some $\tau \in \Z$.
Then, from the symmetry of $s$ and $\tau$,
we can find pairs of rational points $P_{s, \tau^{2}}, P_{\tau, s^{2}} \in E_{s, \tau^{2}}(\Q)$.
Moreover, \cref{unramified_at_p} shows that
if $p$ divides exactly one of $s$ and $\tau$, $p^{n+1}$ divides $s\tau$, and $E_{s, \tau^{2}}(\Q_{p})[p] = 0$,
then $P_{s, \tau^{2}}, P_{\tau, s^{2}} \in p^{n}E_{s, \tau^{2}}(\Q)$.
On the other hand,
Fujita and Terai \cite[Theorem 1.5(1)]{Fujita-Terai} proved that these two rational points can be extended to a system of generators of $E_{s, \tau^{2}}(\Q)$
whenever $s^{4}+\tau^{4}$ is fourth-power-free.
As a consequence,
the same argument in the proof of \cref{main_result}
with a replacement of \cref{Dainobu_criterion} to \cref{Dainobu_multiplicity}
implies the assertion.
\end{proof}

\section{Proof of \cref{infinite}}

Let $l \geq 5$ be a prime number,
$E = E^{(l)}$ be an elliptic curve defined by $y^{2} = x^{3}-lx$,
$\mathcal{E}$ be the N\'eron model of $E$ over $\Z_{l}$ (see s.g.\ \cite[Ch.\ IV]{Silverman_advanced}),
$\bar{\mathcal{E}}$ be the special fiber of $\mathcal{E}$,
$\bar{\mathcal{E}}^{0}$ be the identity component of $\mathcal{E}$,
and $c(\bar{\mathcal{E}}) = [\bar{\mathcal{E}} : \bar{\mathcal{E}}^{0}]$.
Then,
since the minimal discriminant of $E$ is $2^{6} \cdot l^{3}$ and we assume that $l \geq 5$,
the table of reduction types of elliptic curves
in \cite[p.\ 365]{Silverman_advanced} shows that
$\mathcal{E}$ is of Type III,
hence $c(E/\Q_{l}) = 2$.
Therefore, (the proof of) the N\'eron-Ogg-Shafarevich criterion \cite[Theorem 1]{Serre-Tate} implies that
for every prime number $p \neq 2, l$,
the Galois extension $\Q(E[p])/\Q$ is ramified at $v = 2, l$ ($p$, and $\infty$) and unramified outside $2, l, p$ and $\infty$.

\begin{proof} [Proof of \cref{infinite}]
In view of the above argument,
it is sufficient to confirm that
there exists infinitely many prime numbers $l$ of the form $l = s^{4}+t^{2}$
such that $s \equiv 0 \bmod{p^{n+1}}$ or $t \equiv 0 \bmod{p^{n+1}}$.
This is a consequence of a refinement of \cite[Theorem 1]{Friedlander-Iwaniec}
(see \cite[p. 947]{Friedlander-Iwaniec}).
This completes the proof of \cref{infinite}.
\end{proof}

In what follows,
we review the outline of the proof of \cite[Theorem 1]{Serre-Tate}
for the convenience of the reader.

\subsection{The mod $p$ N\'eron-Ogg-Shafarevich criterion for abelian varieties \`a la Serre-Tate \cite{Serre-Tate}}

The N\'eron-Ogg-Shafarevich criterion for abelian varieties
(defined over a finite extension of the field $\Q_{v}$ of $v$-adic numbers)
is usually stated as follows:

\begin{theorem} [{\cite[Theorem 1]{Serre-Tate}}] \label{usual_NOS}
Let $K$ be a finite extension of $\Q_{v}$
and $A$ be an abelian variety defined over $K$.
Then, the following conditions are equivalent to each other:
\begin{enumerate}
\item
The abelian variety $A$ has good reduction.

\item
For every integer $N$ prime to $v$,
the module $A[N]$ of $N$-torsion points in $A(\overline{K})$ is an unramified $\Gal(\overline{K}/K)$-module.

\item
There exist infinitely many integers $N$ prime to $v$
such that $A[N]$ is an unramified $\Gal(\overline{K}/K)$-module.

\item
For every/some prime number $p$ different from $v$,
the $p$-adic Tate module $T_{p}(A) := \varprojlim_{n} A[p^{n}]$ of $A$
is an unramified $\Gal(\overline{K}/K)$-module.
\end{enumerate}
\end{theorem}

However, the proof of Serre and Tate in \cite{Serre-Tate} actually shows a finer statement.
Let $\mathcal{A}$ be the N\'eron model of $A$ (\cite{Neron}) over the ring $\mathcal{O}_{K}$ of integers in $K$,
$\tilde{\mathcal{A}}$ be the special fiber of $\mathcal{A}$,
$\tilde{\mathcal{A}}^{0}$ be the identity component of $\mathcal{A}$,
and $c(\tilde{\mathcal{A}}) = [\tilde{\mathcal{A}} : \tilde{\mathcal{A}}^{0}]$.
Then, a refinement of \cref{usual_NOS} can be stated as follows:

\begin{theorem} \label{refined_NOS}
The following condition is equivalent to each of the conditions in \cref{usual_NOS}:
\begin{enumerate}
\item[(5)]
There exists an integer $N$ prime to $v$ and $c(\tilde{\mathcal{A}})$
such that $A[N]$ is an unramified $\Gal(\overline{K}/K)$-module.
\end{enumerate}
\end{theorem}

In the proof of \cref{infinite},
we used \cref{refined_NOS} for $N = p$.

The proof of \cref{refined_NOS} is given in a complete form
in the proof of \cite[Theorem 1]{Serre-Tate}.
In fact, the implication $(1) \Rightarrow (2), (4)$ is well-known before \cite{Serre-Tate}.
However, the key implication in the proof of \cref{infinite} is $(5) \Rightarrow (1)$
(i.e., $\neg (1) \Rightarrow \neg (5)$),
which implies that $\Q(E_{s, t}[p])/\Q$ is actually ramified at $l = s^{4}+t^{2}$. 
In what follows,
we review the outline of the proof of the implication $\neg (1) \Rightarrow \neg (5)$
following Serre and Tate \cite{Serre-Tate}
for the convenience of the reader.

The key lemmas for the proof of \cref{refined_NOS} are the following:

\begin{lemma} [{cf.\ \cite{Barsotti,Chevalley,Conrad_Chevalley,Rosenlicht}}] \label{Chevalley}
\footnote{
	According to the introduction of \cite{Rosenlicht},
	\cref{Chevalley} was first announced by Chevalley in 1953.
	}
Let $k$ be a field and $G$ be a connected algebraic group scheme over $k$.
Then, there exists a connected linear algebraic (i.e., affine) normal subgroup scheme $H$ of $G$
such that $G/H$ is an abelian (i.e., projective) variety.
\end{lemma}

\begin{lemma} [{\cite[Lemma 1]{Serre-Tate}}] \label{vanishing_cycle}
Let $k$ be a field,
$\tilde{\mathcal{A}}$ be a commutative algebraic group over $k$,
and $\tilde{\mathcal{A}}^{0}$ be the connected component of $\tilde{\mathcal{A}}$.
Suppose that $\tilde{\mathcal{A}}^{0}$ is an extension of an abelian variety $B$ by an algebraic subgroup $H$ of $\tilde{\mathcal{A}}$
and $H \simeq T \times U$ for an algebraic torus $T$ and a unipotent algebraic group $U$.
Let $c(\tilde{\mathcal{A}}) := [\tilde{\mathcal{A}} : \tilde{\mathcal{A}}^{0}]$ be the index of $\tilde{\mathcal{A}}^{0}$ in $\tilde{\mathcal{A}}$.
Suppose that $k$ is perfect and $N$ is an integer prime to $\Char(k)$.
Then, the $\Z/N\Z$-module $\tilde{\mathcal{A}}(\overline{k})[N]$ is isomorphic to an extension of a cyclic group of order dividing $c(\tilde{\mathcal{A}})$ by a free $\Z/N\Z$-module of rank $\dim T+2\dim B$.
\end{lemma}

\begin{lemma} [{\cite[Lemma 2]{Serre-Tate}}] \label{Hensel_isom}
Let $K$ be a finite extension of $\Q_{v}$,
$\mathcal{O}$ be the ring of integers of $K$,
$\mathfrak{m}$ be its maximal ideal,
$k = \mathcal{O}/\mathfrak{m}$ be the residue field of $\mathcal{O}$,
and $I$ be the inertia subgroup of the Galois group $G_{K} := \Gal(\overline{K}/K)$.
Let $A$ be an abelian variety defined over $K$
and $\mathcal{A}$ be its N\'eron model.
Let $N$ be an integer prime to $v$.
Then the modulo $\mathfrak{m}$ reduction map defines an isomorphism
\[
	A[N]^{I}
	= A(\overline{K})[N]^{I}
	\simeq \mathcal{A}(\overline{\mathcal{O}})[N]^{I}
	\simeq \tilde{\mathcal{A}}(\overline{k})[N]
\]
which is compatible with the action of the group $G_{K}/I \simeq \Gal(\overline{k}/k)$.
Here, $\overline{O}$ denotes the ring of integers in $\overline{K}$.
\end{lemma}

The proof of $(1) \Rightarrow (2)$ is easy.
Suppose that $A$ has good reduction.
Then, \cref{vanishing_cycle} shows that
the $\Z/N\Z$-modules $A(\overline{K})[N]$ and $\tilde{\mathcal{A}}(\overline{O})[N]$ are free of rank $2\dim A = 2\dim \tilde{\mathcal{A}}$.
Hence, \cref{Hensel_isom} implies an equality $A(\overline{K})[N] = A(\overline{K})[N]^{I}$,
which means that $A(\overline{K})[N]$ is an unramified $\Gal(\overline{K}/K)$-module.
This completes the proof of the implication $(1) \Rightarrow (2)$.

For the proof of $(5) \Rightarrow (1)$,
suppose that there exists an integer $N$ prime to $v$ and $c(\tilde{\mathcal{A}})$
such that $A(\overline{K})[N]$ is an unramified $\Gal(\overline{K}/K)$-module.
First, note that
\cref{Chevalley} implies that
there exist an abelian variety $B$, an algebraic torus $T$, and and unipotent algebraic group $U$ 
such that $\tilde{\mathcal{A}}$ is an isomorphic to an extension of $B$ by $T \times U$ over the residue field $k$ of $K$.
In particular, we have an inequality $\dim \tilde{\mathcal{A}} \geq \dim B + \dim T$.
On the other hand, \cref{Hensel_isom} shows
an isomorphism $A(\overline{K})[N] \simeq \tilde{\mathcal{A}}(\overline{O})[N]$ of $\Z/N\Z$-modules.
Since we assume that $N$ is prime to $v = \Char k$ and $c(\tilde{\mathcal{A}})$,
\cref{vanishing_cycle} shows that the $\Z/N\Z$-module $A(\overline{K})[N]$
is free of rank $2 \dim A = \dim T + 2\dim B$.
This implies that $\dim T = 0$, hence $\dim A = \dim B = \dim \tilde{\mathcal{A}}$.
In particular, $\tilde{A} \simeq B$ is an abelian variety,
especially a proper scheme over $k$.
Finally, a purely scheme theoretic lemma \cite[Lemma 3]{Serre-Tate} shows that
the N\'eron model $\mathcal{A}$ of $A$ is a proper scheme over the ring $\mathcal{O}$ of integers in $K$,
hence $A$ has good reduction \cite{Neron}.
This completes the proof of the implication $(5) \Rightarrow (1)$.

\section{Upper bounds for the Mordell-Weil rank of $E_{s, t}$}

By a standard descent method for 2-isogeny,
we can prove the following,
which itself should be well-known for experts.
(See e.g.\ \cite[Ch.\ X, Propositions 6.1 and 6.2]{Silverman_AEC} for similar results.)

\begin{theorem} \label{upper_bound}
Let $l$ be a non-zero integer
and $E^{(l)}$ be an elliptic curve defined by $y^{2} = x^{3}-lx$.
Suppose that $l$ is a prime number.
Then, the following inequality holds.
\[
	\rank E^{(l)}(\Q)
	\leq \begin{cases}
	0 & \text{if $l \equiv 3, 11, 13 \bmod{16}$}, \\
	1 & \text{if $l \equiv 2, 5, 7, 9, 15 \bmod{16}$}, \\
	2 & \text{if $l \equiv 1 \bmod{16}$}.
	\end{cases}
\]
\end{theorem}

Since the elliptic curve $E_{s, t} : y^{2} = x^{3}-(s^{4}+t^{2})x$ in \cref{primitivity}
has an obvious rational point $P_{s, t} = (-s^{2}, st)$ whose order is infinite,
we obtain the following:

\begin{corollary} \label{rank_one}
Let $s, t$ be integers
and $E_{s, t}$ be an elliptic curve defined by $y^{2} = x^{3}-(s^{4}+t^{2})x$.
Suppose that $s^{4}+t^{2}$ is a prime number such that $s$ is even and $t \equiv \pm 3 \bmod{8}$.
Then, the equality $\rank E_{s, t}(\Q) = 1$ holds.
\end{corollary}

\begin{remark}
The preceding works
\cite{Sairaiji-Yamauchi_JNT,Sairaiji-Yamauchi_JTNB,Hiranouchi}
obtained some lower bounds for the orders of the $p$-parts of the ideal class groups $\Cl(\Q(E[p^{n}]))$ of the $p^{n}$-division fields $\Q(E[p^{n}])$ of elliptic curves $E$.
However, their general results fail to give any non-trivial bounds when $\rank E(\Q) \leq 1$.
Therefore, it is natural to ask whether
we can obtain some lower bounds for the order of $\Cl(\Q[p])[p]$
for any elliptic curve $E$ such that $\rank E(\Q) \leq 1$.
We can give an affirmative answer to this question
from \cref{main_result,rank_one}.
Indeed, we can obtain an infinite family of elliptic curves $E_{s,t}$
such that $\rank E_{s,t}(\Q) = 1$ and $\Cl(\Q(E_{s,t}[p]))[p] \neq 0$.
In fact, \cref{infinite} ensures that
the set of such number fields $\Q(E_{s,t}[p])$ contains infinitely isomorphism classes.
\end{remark}

\subsection{Proof of \cref{upper_bound}}

For the convenience of the reader,
we give a proof of \cref{upper_bound} along the descent method in \cite[Ch.\ X]{Silverman_AEC}.

Let $\phi : E^{(l)} \to E^{(-4l)}$ be a rational map defined by $\phi(x, y) := (y^{2}/x^{2}, -y(x^{2}+l)/x^{2})$.
Then, it is a 2-isogeny whose kernel is generated by the point $(0, 0)$ of order 2.
Let $\hat{\phi}$ be the dual isogeny of $\phi$.
Then, we have an exact sequence with natural maps
\begin{align*}
	0
	&\to E^{(-4l)}(\Q)[\hat{\phi}]/\phi(E^{(l)}(\Q)[2])
		\to E^{(-4l)}(\Q)/\phi(E^{(l)}(\Q)) \\
	&\xrightarrow{\hat{\phi}} E^{(l)}(\Q)/2E^{(l)}(\Q)
		\to E^{(l)}(\Q)/\hat{\phi}(E^{(-4l)}(\Q))
			\to 0.
\end{align*}
Since $E^{(-4l)}(\Q)[\hat{\phi}] = \{ \infty, (0, 0) \}$ and $\phi(E^{(l)}(\Q)[2]) = \phi(\{ \infty, (0, 0) \}) = \{ \infty \}$,
we have
\begin{align*}
	\rank E^{(l)}(\Q)
	&= \dim_{\F_{2}} (E^{(l)}(\Q)/2E^{(l)}(\Q))-\dim_{\F_{2}} E^{(l)}(\Q)[2] \\
	&= \dim_{\F_{2}} (E^{(-4l)}(\Q)/\phi(E^{(l)}(\Q)))
		- \dim_{\F_{2}} (E^{(-4l)}(\Q)[\hat{\phi}]/\phi(E^{(l)}(\Q)[2])) \\
	&\quad + \dim_{\F_{2}} (E^{(l)}(\Q)/\hat{\phi}(E^{(-4l)}(\Q)))
		- 1 \\
	&\leq \dim_{\F_{2}} \Sel(\Q, E^{(l)}[\phi]) + \dim_{\F_{2}} \Sel(\Q, E^{(-4l)}[\hat{\phi}]) - 2.
\end{align*}
Here, recall that the $\phi$-Selmer group $\Sel(\Q, E^{(l)}[\phi])$ ($= S^{(\phi)}(E^{(l)}/\Q)$ in \cite{Silverman_AEC}) can be identified with
the $\F_{2}$-vector space consisting of the isomorphism classes of principal homogeneous spaces for $E^{(l)}$
which are trivialized by $\phi$ and have $\Q_{v}$-rational points at every place $v$.
Such classes are parametrized by $d\Q^{\times 2} \in \Q^{\times}/\Q^{\times 2}$
and has a representative given by
\[
	D^{(l)}_{d} : dX_{2}^{2} = d^{2}X_{0}^{4} + 4lX_{1}^{4}.
\]
Thus, we can check that
\footnote{
	Indeed,
	since $[X_{0} : X_{1} : X_{2}] = [0 : 1 : 2] \in D^{(l)}_{l}(\Q)$,
	it is sufficient to check that
	$D^{(l)}_{-1}, D^{(l)}_{2},  D^{(l)}_{-2} \not\in \Sel(\Q, E^{(l)}[\phi])$,
	which follows from
	$D^{(l)}_{-1}(\R) = D^{(l)}_{-2}(\R) = \emptyset$ for every $l$,
	and $D^{(l)}_{2}(\Q_{2}) = \emptyset$ for $l \not\equiv \pm1, 2, 7 \bmod{16}$.
	}
\[
	\{ D^{(l)}_{1}, D^{(l)}_{l} \}
	\subset \Sel(\Q, E^{(l)}[\phi])
	\subset \begin{cases}
		\{ D^{(l)}_{1}, D^{(l)}_{l} \} & \text{if $l \not\equiv \pm1, 7 \bmod{16}$}, \\
		\{ D^{(l)}_{1}, D^{(l)}_{2}, D^{(l)}_{l}, D^{(l)}_{2l} \} & \text{if $l \equiv \pm1, 7 \bmod{16}$}.
	\end{cases}
\]
In a similar manner,
we can check that
\footnote{
	Indeed,
	since $[0 : 1 :  4] \in D^{(-4l)}_{-l}(\Q)$,
	it is sufficient to check that
	$D^{(-4l)}_{2}(\Q_{2}) = D^{(-4l)}_{2l}(\Q_{2}) = \emptyset$ for odd $l$,
	$D^{(-4l)}_{-1}(\Q_{2}) = \emptyset$ for $l \equiv 3 \bmod{4}$ and $l \equiv 13 \bmod{16}$.
	}
\[
	\{ D^{(-4l)}_{1}, D^{(-4l)}_{- l} \}
	\subset \Sel(\Q, E^{(-4l)}[\hat{\phi}])
	\subset \begin{cases}
		\{ D^{(-4l)}_{\pm 1}, D^{(-4l)}_{\pm l} \} & \text{if $l \equiv 1, 2, 5, 9 \bmod{16}$}, \\
		\{ D^{(-4l)}_{1}, D^{(-4l)}_{- l} \} & \text{if $l \equiv 3, 7, 11, 13, 15 \bmod{16}$}.
	\end{cases}
\]
As a conclusion,
we obtain the desired bound.

\begin{remark}
In fact, we can determine both groups $\Sel(\Q, E^{(l)}[\phi])$ and $\Sel(\Q, E^{(-4l)}[\hat{\phi}])$ completely.
For the former,
we have the equality
\[
	\Sel(\Q, E^{(l)}[\phi]) = \{ D^{(l)}_{1}, D^{(l)}_{2}, D^{(l)}_{l}, D^{(l)}_{2l} \}
	\quad \text{if $l \equiv \pm1, 7 \bmod{16}$},
\]
which is a consequence of the following facts.
\begin{itemize}
\item
$[1 : 1 : \sqrt{2(l+1)}] \in D^{(l)}_{2}(\Q_{2})$ for $l \equiv 1 \bmod{16}$,
$[(8-l)^{1/4} : 1 : 4] \in D^{(l)}_{2}(\Q_{2})$ for $l \equiv 7 \bmod{16}$,
and $[(-l)^{1/4} : 1 : 0] \in D^{(l)}_{2}(\Q_{2})$ for $l \equiv 15 \bmod{16}$.

\item
$[1 : 0 : \sqrt{2}] \in D^{(l)}_{2}(\Q_{l})$ for $l \equiv \pm1 \bmod{8}$.

\item
$[1 : 0 : \sqrt{2}] \in D^{(l)}_{2}(\R)$ for $l > 0$.
\end{itemize}
For the latter,
we have the equality
\[
	\Sel(\Q, E^{(-4l)}[\hat{\phi}]) = \{ D^{(-4l)}_{\pm 1}, D^{(-4l)}_{\pm l} \}
	\quad \text{$l \equiv 1, 2, 5, 9 \bmod{16}$},
\]
which is a consequence of the following facts.
\begin{itemize}
\item
$[2 : 1 : 4] \in D^{(-8)}_{-1}(\Q)$ for $l = 2$.

\item
$[0 : 1 : 4\sqrt{l}] \in D^{(-4l)}_{-1}(\Q_{2})$ for $l \equiv 1 \bmod{8}$
and $[2(l-4)^{1/4} : 1 : 8] \in D^{(-4l)}_{-1}(\Q_{2})$ for $l \equiv 5 \bmod{16}$.

\item
$[1 : 0 : \sqrt{-1}] \in D^{(-4l)}_{-1}(\Q_{l})$ for $l \equiv 1 \bmod{4}$.

\item
$[0 : 1 : 4\sqrt{l}] \in D^{(-4l)}_{-1}(\R)$ for $l > 0$.
\end{itemize}
\end{remark}

\section*{Acknowledgements}

The authors would like to thank Prof.\ Kenichi Bannai for his careful reading of a draft of this paper.
The authors would like to thank Yoshinori Kanamura for valuable comments on a draft of this paper and giving information on references.
The authors would like to thank Prof.\ Masato Kurihara for his careful reading and valuable comments on a draft of this paper.

\end{document}